\documentclass{amsart}

\newtheorem{theorem}{Theorem}[section]
\newtheorem{lemma}[theorem]{Lemma}

\theoremstyle{definition}

\newtheorem{pro}[theorem]{Proposition}
\newtheorem{cor}[theorem]{Corollary}

\theoremstyle{remark}
\newtheorem{rem}[theorem]{Remark}

\numberwithin{equation}{section}

\begin{document}

\title[Minimally almost periodic group topology on Abelian groups]{Minimally almost periodic group topology on infinite countable Abelian groups: A solution to Comfort's problem}

\author{S.S. Gabriyelyan}
\address{Department of Mathematics, Ben-Gurion University of the Negev, Beer-Sheva P.O. 653, Israel}
\email{saak@math.bgu.ac.il}
\thanks{The author was partially supported  by Israel Ministry of Immigrant Absorption.}

\subjclass[2000]{Primary 22A10, 43A40; Secondary 54H11}



\keywords{Countable group, characterized group, $T$-sequence, dual group, von Neumann radical}

\begin{abstract}
For any countable subgroup $H$ of an unbounded Abelian group $G$ there is a complete  Hausdorff group topology $\tau$ on $G$ such that $H$ is the von Neumann radical of $(G,\tau)$. In particular, we obtain the positive answer to Comfort's question: any unbounded countable Abelian group admits a complete Hausdorff minimally almost periodic (MinAP) group topology. A bounded  infinite Abelian group  admits a MinAP group topology if and only if all its leading Ulm-Kaplansky invariants are infinite. If, in addition,  $G$ is countably infinite, a MinAP group topology can be chosen to be complete.
\end{abstract}

\maketitle

\section{Introduction}

{\bf I. Questions and History.} For an Abelian topological group $G$, $G^{\wedge}$ denotes the group of all continuous characters on $G$ endowed with the compact-open topology.  Denote by $\mathbf{n}(G) = \cap_{\chi\in G^{\wedge}} {\rm ker} \chi$ the von Neumann radical of $G$. Following  von Neumann \cite{Neu}, if $\mathbf{n}(G) = G$, the group $G$ is called minimally almost periodic (MinAP), and $G$ is called maximally almost periodic (MAP) if $\mathbf{n}(G) = 0$.

Let $G$ be an infinite Abelian group. Denote  by $\mathcal{NR}(G)$ [respectively $\mathcal{NRC}(G)]$ the set of all subgroups $H$ of $G$  for which there exists a [respectively complete] non-discrete Hausdorff group topology $\tau$ on $G$ such that the von Neumann radical of $(G,\tau)$ is $H$, i.e., $\mathbf{n} (G,\tau)=H$. It is clear that $\mathcal{NRC}(G)\subseteq \mathcal{NR}(G)$.

Let $G$ be an Abelian topological group. The richness of its dual group $G^\wedge$ is one of the most important properties of $G$. The von Neumann radical measures this richness, thus the following general question is important:
\begin{enumerate}
\item[] {\bf Question 1.} (Problem 2 \cite{Ga2}) {\it Describe the sets $\mathcal{NR}(G)$ and $\mathcal{NRC}(G).$}
\end{enumerate}

The following proposition (proved in Section \ref{s3}) is a simple corollary of the main result of \cite{CRT}:
\begin{pro} \label{p00}
{\it Every infinite Abelian group admits a  complete non-trivial Hausdorff group topology with  trivial von Neumann radical.}
\end{pro}
Thus, the trivial group $\{ 0\}$ belongs to $\mathcal{NRC}(G)$ for every infinite Abelian group $G$, i.e., $\mathcal{NRC}(G)$ is always not empty. A much deeper question is whether any infinite Abelian group admits a Hausdorff group topology with  {\it non-zero} von Neumann radical. The positive answer was given by Ajtai, Havas and Koml\'{o}s \cite{AHK}. Using the method of $T$-sequences, Protasov and Zelenyuk \cite{ZP1} proved that every infinite Abelian group admits a {\it complete}  Hausdorff group topology for which characters do not separate points, i.e., $\mathcal{NRC}(G)\not= \{ \{ 0\}\}$ for every infinite $G$.

Question 1 has two interesting extreme cases. Denote by $\mathcal{S}(G)$ the set of all subgroups of an infinite Abelian group $G$.
\begin{enumerate}
\item[] {\bf Question 2.} (Problem 4 \cite{Ga2}) {\it Describe all infinite Abelian groups $G$ such that $\mathcal{NR}(G)= \mathcal{S}(G)$ [respectively $\mathcal{NRC}(G)=\mathcal{S}(G)$].}
\end{enumerate}
The second and, perhaps, the most interesting special case of Question 1 is the following:
\begin{enumerate}
\item[] {\bf Question 3.} (Problem 5 \cite{Ga2}) {\it Describe all infinite Abelian groups $G$ such that $G\in \mathcal{NR}(G)$ (or $G\in \mathcal{NRC}(G)$).}
\end{enumerate}

Note that the fact of existence of Abelian groups admitting no nonzero continuous characters is known long ago. It is known \cite{HZ, Sm}  that each character of a linear topological space $L$ over $\mathbb{R}$, where $L$ is regarded as an Abelian topological group under +, is defined by a continuous linear functional on $L$. So, if $L$ does not admit nontrivial continuous functionals, then it is a MinAP group. The most famous such spaces are $L^p$ with $0<p<1$ \cite{Day}, another examples see  \cite[23.32]{HR1}.  Nienhuys \cite{Ni1} showed the existence of a metric solenoidal monothetic MinAP group (see also \cite{Ni2}). In \cite{AHK} the authors proved the main result showing that the following groups admit a MinAP group topology: 1) the group $\mathbb{Z}$ of integers; 2) the Pr\"{u}fer group $\mathbb{Z}(p^\infty)$, where $p$ is a prime number; 3) direct sums $\bigoplus_{i=1}^\infty \mathbb{Z} (p_i)$ of cyclic groups $\mathbb{Z} (p_i)$ of order $p_i$, where all $p_i$ are prime numbers and either $p_1=p_2=\dots$ or $p_1 <p_2 <\dots$.  Prodanov \cite{Pro} gave an elementary example of a MinAP group. Protasov \cite{Pr1} posed the question whether every infinite Abelian group admits a minimally almost periodic group topology. Using a result of Graev \cite{Gra}, Remus \cite{Rem} proved that for every natural number $n$ there exists a connected MinAP group which is algebraically generated by elements of order $n$. On the other hand, he gave (see \cite{Com}) a simple example of a bounded group $G$ which does not admit any Hausdorff group topology $\tau$ such that $(G,\tau)$ is minimally almost periodic. So, in general, for bounded groups the answer to Protasov's question is negative. This justifies the following problems:

\begin{itemize}
\item[] {\bf Question 4.} (Comfort's Problem 521 \cite{Com}) {\it Does every Abelian group which is not of bounded exponent admit a minimally almost periodic topological group topology? What about
    the countable case?}
\item[] {\bf Question 5.} (Question 2.6.1 \cite{ZP2}) {\it Let $G$ be a torsion free countable group. Does there exist a Hausdorff group topology on $G$ with only zero character?}
\end{itemize}
Corollary \ref{cc2} (see below) justifies the following question:
\begin{itemize}
\item[] {\bf Question 6.} {\it Let a subgroup $H$ of an Abelian group $G$ belong to $\mathcal{NR}(G)$ [respectively $\mathcal{NRC}(G)$]. Is it true that $\mathcal{S}(H) \subset \mathcal{NR}(G)$ [respectively $\mathcal{S}(H) \subset \mathcal{NRC}(G)$]? In particular, let an unbounded Abelian group $G$ admit a [respectively complete] minimally almost periodic group topology. Is it true that $\mathcal{NR}(G) = \mathcal{S}(G)$ [respectively $\mathcal{NRC}(G) = \mathcal{S}(G)$]?}
\end{itemize}
In other words, in those two particle cases Corollary \ref{cc2} shows an equivalence of Questions 2 and 3. So, if the answer to Comfort's question is ``yes'', then Question 6 also has the positive answer. In such a case we obtain that, if an Abelian group $G$ admits a MinAP group topology, then for any its subgroup $H$ there is a Hausdorff group topology $\tau$ on $G$ such that $\mathbf{n}(G,\tau)=H$.

Taking into account that any Abelian group can be embed into a compact one,  Dikranjan asked:
\begin{itemize}
\item[] {\bf Question 7.} (Dikranjan) {\it Let an unbounded  Abelian group $G$ admit a compact group topology. Does it admit also a MinAP group topology? What about the groups of the form $\prod_{(n,p) \in E} \left(\mathbb{Z} (p^{n} ) \right)^{k_{n,p}}$ or $\Delta_p \times F$, where $\Delta_p$ is the groups of $p$-adic integers,  $F$ is a finite group, $E$ is an infinite subset of $\mathbb{N}\times P$, $P$ is the set of all primes and all $k_{n,p}$ are arbitrary positive integers?}
\end{itemize}
We also select the next problem:
\begin{itemize}
\item[] {\bf Question 8.} {\it Does every uncountable indecomposable Abelian group admit a MinAP group topology?}
\end{itemize}
Note that, by Baer's theorem  \cite[43.1]{Fuc}, $\Delta_p$ is continual and indecomposable.

Another important special and non-trivial case of Question 1 is when does $\mathcal{NR}(G)$ contain finite non-trivial groups. It was mentioned, without a proof, by Dikranjan, Milan and Tonolo \cite{DMT}, that the Pr\"{u}fer group $\mathbb{Z}(p^\infty)$ may admit such a topology. Luk\'{a}cs \cite{Luk} gave a proof of this fact (for every prime $p>2)$ and  called such groups almost maximally almost-periodic (AMAP). He proved also that infinite direct sums of finite groups are AMAP. These results were generalized in \cite{Ngu}.
Obviously, an AMAP group contains a non-trivial finite subgroup, so it cannot be torsion free. Therefore, an Abelian group admitting an AMAP group topology is not torsion free. It turns out that the converse assertion is also true:
\begin{itemize}
\item[] {\bf Theorem A.} \cite{Ga2} {\it An Abelian group admits an AMAP group topology iff it is not torsion free.}
\end{itemize}

The exponent of an Abelian group $G$ is denoted by $\exp G$.
A complete characterization of those finitely generated subgroups of an infinite Abelian group $G$ which are the von Neumann radical for some Hausdorff group topology on $G$ is given in \cite{Ga3}:
\begin{itemize}
\item[] {\bf Theorem B.} \cite{Ga3} {\it Let $H$ be a finitely generated subgroups of an infinite Abelian group $G$.
    \begin{enumerate}
    \item If $\exp G=\infty$, then $H\in \mathcal{NRC}(G)$.
    \item If $\exp G<\infty$, then $H\in \mathcal{NR}(G)$ iff $H\in \mathcal{NRC}(G)$ iff $G$ contains a subgroup of the form $\mathbb{Z} (\exp (H))^{(\omega)}$.
    \end{enumerate} }
\end{itemize}

{\bf II. Main results.} In this paper we give a complete answer to Question 1 for bounded groups and countably infinite unbounded ones. The following three theorems  are main in the article:
\begin{theorem} \label{t2}
Let $G$ be an infinite bounded Abelian group and $H$ its nonzero subgroup. Then $H\in \mathcal{NR}(G)$ if and only if the group $G$ contains a subgroup of the form $\mathbb{Z} (\exp (H))^{(\omega)}$. In the particular case when $H$ is countable, a topology $\tau$ such that $\mathbf{n}(G,\tau) =H$ can be chosen to be complete.
\end{theorem}

\begin{theorem} \label{t3}
Let $G$ be an unbounded Abelian group. Then each bounded subgroup $H$ of $G$ belongs to $\mathcal{NR} (G)$. Moreover, if $H$ is countable, then $H\in \mathcal{NRC} (G)$.
\end{theorem}

\begin{theorem} \label{t1}
Let $G$ be an unbounded Abelian group. For any countable subgroup $H$ of $G$ there exists a complete Hausdorff group topology such that $H=\mathbf{n} (G)$, i.e.,  $H\in \mathcal{NRC} (G)$. In particular, if $G$ is countably infinite, then $\mathcal{NRC}(G) = \mathcal{S}(G)$.
\end{theorem}

As an evident corollary of these theorems we obtain a complete characterization of countably infinite Abelian groups which admit a MinAP group topology, in particular, we give the positive answer to Comfort's Question 4 for the countable case and, hence, to Question 5.

\begin{cor} \label{c1}
{\it Each unbounded countably infinite Abelian group
 admits a complete Hausdorff minimally almost periodic group topology.}
\end{cor}
Let $G$ be a bounded group. Then $G$ has the form $G =\bigoplus_{p\in M} \bigoplus_{i=1}^{n_p} \mathbb{Z}(p^i)^{(\kappa_{i,p})}$, where $M$ is a finite set of prime numbers. Leading Ulm-Kaplansky invariants of $G$ are the cardinal numbers $\{ \kappa_{n_p ,p} \}_{ p\in M}$. For bounded groups we have the following:
\begin{cor} \label{c2}
{\it An infinite bounded Abelian group $G$ admits a MinAP group topology if and only if all its leading Ulm-Kaplansky invariants are infinite. In such a case $\mathcal{NR}(G) = \mathcal{S}(G)$.}
\end{cor}
Moreover, we have the following corollary of Theorems \ref{t2}-\ref{t1}:
\begin{cor} \label{cc2}
{\it Let  $H$ be a bounded [respectively  countably infinite] subgroup of an Abelian group $G$. If $H$ belongs to $\mathcal{NR}(G)$, then $\mathcal{S}(H) \subset \mathcal{NR}(G)$ [respectively $\mathcal{S}(H) \subset \mathcal{NRC}(G)$]. }
\end{cor}

For the unbounded uncountable case Comfort's Problem as well as the uncountable counterpart of Question 5 remain open. We can prove only the following:
\begin{theorem} \label{t4}
Let $G= \bigoplus_{\alpha \in I} G_\alpha$, where $I$ is an arbitrary set of indices and  $G_\alpha$  are (nonzero) countable Abelian groups. If $G$ is unbounded, then it admits a MinAP group topology.
\end{theorem}

Since a divisible group is a direct sum of full rational groups $\mathbb{Q}$ and groups of the form $\mathbb{Z} (p^\infty)$  \cite[Theorem 19.1]{Fuc}, we obtain the following (see \cite{Rem}):
\begin{cor} \label{c3}
{\it {\rm 1. \cite{Rem}} Every free Abelian group admits a MinAP group topology.

{\rm 2. \cite{Rem}} Each divisible Abelian group admits a MinAP group topology.

{\rm 3. } Every linear space over $\mathbb{R}$ or $\mathbb{C}$ admits a MinAP group topology.
}
\end{cor}

Since any Abelian group embeds into divisible one  \cite[A.15]{HR1}, we obtain:
\begin{cor} \label{c4}
{\it Each  Abelian group $H$ embeds into an Abelian group  $G$  that admits a
MinAP group  topology.}
\end{cor}
Note that it is known the general result \cite{DW}: every Hausdorff topological  group can be embedded into a MinAP one as a closed subgroup.

Taking into consideration an algebraic structure of connected locally compact Abelian groups  \cite[Theorem 25.23]{HR1}, we obtain a particular answer to Dikranjan's question:
\begin{cor} \label{c5}
{\it If an Abelian group $G$ admits a connected locally compact group topology, then it admits also a MinAP group  topology. In particular, the group  $\mathbb{R}$  of the reals and
 the circle group  $\mathbb{T}$   admit a MinAP group topology.}
\end{cor}
In Remark \ref{r1} we show that  to give the positive answer to Question 7 we have to obtain the positive solutions to the special groups in this question.

\section{The scheme of the proofs.} \label{s1}

In what follows we need some notions and notations.

The order of an element $g\in G$ we denote by $o(g)$. The subgroup of $G$ generated by a subset $A$ is denoted by $\langle A\rangle$.  If $n>0$ we put $(n+1)A :=(n)A +A$.

Let $H$ be a subgroup of a topological Abelian group $X$. The annihilator of $H$  we denote by $H^{\perp} :=\{ \chi \in X^\wedge : (\chi, h)=1 \mbox{ for every } h\in H\}$. $H$ is called dually closed in $X$ if for every $x\in X\setminus H$ there exists a character $\chi\in H^{\perp}$ such that $(\chi,x)\not= 1$. $H$ is called dually embedded in $X$ if every character of $H$ can be extended to a character of $X$.

The following lemma plays an important role.
\begin{lemma} \label{l1}
{\rm \cite{Ga1}} {\it Let $K$ be a dually closed and dually embedded subgroup of a topological group $G$. Then $\mathbf{n} (K) =\mathbf{n} (G)$. }
\end{lemma}
This lemma allows us to reduce the general case to  some particular and simpler ones.

{\bf 1. Reduction Principle}. Let $H$ be a subgroup of an infinite Abelian group $G$ and let $Y$ be an arbitrary subgroup of $G$ containing $H$. If $Y$ admits a Hausdorff group topology $\tau$ such that $\mathbf{n} (Y,\tau) =H$, then we can consider the topology on $G$ in which $Y$ is open. Since any open subgroup is dually closed and dually embedded  \cite[Lemma 3.3]{Nob}, by Lemma \ref{l1}, we obtain that $\mathbf{n} (G)=H$ either. Thus we can reduce our consideration to $Y$ only.

By this reduction principle, we can split the proofs of Theorems \ref{t2}-\ref{t1} into some special cases using two structural theorems (see Section \ref{s2}) and  the following proposition:
\begin{pro} \label{p4}
{\it Let $G= \left(\bigoplus_{\alpha \in I} G_\alpha \right) \times  \prod_{\beta \in J} H_\beta$, where $I$ and $J$ be  arbitrary sets of indices and each group
$G_\alpha$ and $H_\beta$ admit a MinAP group topology. Then $G$ also admits a MinAP group topology.}
\end{pro}

By the reduction principle and Proposition \ref{p4} and taking into account that the von Neumann radical of a  product of topological groups is the product of their von Neumann radicals, in the next section we show (see Theorems \ref{td01} and \ref{td02}) that to prove Theorems \ref{t2} and \ref{t3} we have to prove only the following assertions:
\begin{enumerate}
\item A group of the form $\mathbb{Z} (p^r)^{(\omega)}$ admits a complete MinAP group topology.
\item A group $G$ of the form $\mathbb{Z}\oplus H_0$, where $H_0$ is finite,  admits a complete Hausdorff group topology $\tau$ such that $\mathbf{n}(G,\tau)=H_0$.
\item A group $G$ of the form $\mathbb{Z}(p^\infty) +H_0$, where $H_0$ is finite,  admits a complete Hausdorff group topology $\tau$ such that $\mathbf{n}(G,\tau)=H_0$.
\item A group $G$ of the form  $\bigoplus_{i=0}^\infty (H_i + \langle e_i\rangle)$, where
\begin{enumerate}
\item the independent sequence $\{ e_i\}$ satisfies  either the condition
    \begin{enumerate}
    \item $\exp H = o(e_0)=o(e_1)=\dots $, or
    \item $\exp H \le o(e_0)<o(e_1)<\dots $;
    \end{enumerate}
\item there is $M\le \infty$ such that $H_j$ is a finite nonzero subgroup of $G$ for every $0\le j <M$, and, if $M<\infty$, $H_j = \{ 0\}$ for each $j\ge M$;
\end{enumerate}
admits a complete Hausdorff group topology $\tau$ such that $\mathbf{n}(G,\tau)=H$, where $H=\bigoplus_{i=0}^\infty H_i$.
\end{enumerate}
Assertion (1) is proved in \cite{ZP1} (see also Corollary \ref{co1}). Assertions (2), (3) and (4) for $M<\infty$ follow from Theorem B. In Proposition \ref{p3} we prove assertion (4) for the case $M=\infty$.

{\bf 2. Construction of a topology}. Following  Protasov and Zelenyuk \cite{ZP1, ZP2}, we say that a sequence $\mathbf{u} =\{ u_n \}$ in a group $G$ is a $T$-sequence if there is a Hausdorff group topology on $G$ with respect to which $u_n $ converges to zero. The group $G$ equipped with the finest group topology with this property is denoted by $(G, \mathbf{u})$. We note also that, by
\cite[Theorem 2.3.11]{ZP2}, the group $(G, \mathbf{u})$ is complete.

Following \cite{BDM}, we say that a sequence $\mathbf{u} =\{ u_n\}$ is a $TB$-sequence in a group $G$ if there is a precompact Hausdorff group topology on $G$ in which  $u_n \to 0$.

For a sequence $\{ d_n \}$ and $l,m\in \mathbb{N}$, one puts \cite{ZP1}:
\[
\begin{split}
A(k,m) = \left\{ \right.  &  n_1 d_{r_1} +\dots +n_s d_{r_s} | \; \mbox{where } \\
 &   m\le r_1 <\dots < r_s , n_i \in \mathbb{Z} \setminus \{ 0\} , \sum_{i=1}^s | n_i | \le k+1 \} \cup \{ 0\}.
\end{split}
\]
Then the Protasov-Zelenyuk criterion is formulated as follows:
\begin{itemize}
\item[] {\bf Theorem C.} \cite{ZP1} {\it
A sequence  $\{ d_n \}$ of elements of an Abelian group $G$ is a $T$-sequence if and only if, for every $k\ge 0$ and for every element $g\in G, g\not= 0$, there is a natural number $m$ such that $g\not\in A(k,m)$.}
\end{itemize}

Now the second step of the proofs is as follows: for each special case, we construct a special $T$-sequence $\mathbf{d}$ on $G$ and equip $G$ with the Protasov-Zelenyuk group topology generated by the chosen $T$-sequence $\mathbf{d}$.

In the paper  we consider only groups of the form $(G, \mathbf{u})$, where $\mathbf{u}$ is a $T$-sequence. Since, by definition, $(G, \mathbf{u})$ has a countable open subgroup (for example $\langle \mathbf{u} \rangle$), we are forced to restrict ourselves to countable groups only. So, in the general setting, the method of $T$-sequences does not work. We can apply this method for  bounded uncountable groups since any bounded group is a direct sum of finite ones.

{\bf 3. Computation of the von Neumann kernel.} Let $X$ be an Abelian topological group and $\mathbf{u} =\{ u_n \}$ a sequence of elements of $X^{\wedge}$. Following Dikranjan et al. \cite{DMT}, we denote by $s_{\mathbf{u}} (X)$ the set of all $x\in X$ such that $(u_n , x)\to 1$. A group $G$ with the discrete topology is denoted by $G_d$. The following theorem allows us to compute  the von Neumann kernel:

\begin{itemize}
\item[] {\bf Theorem D.} \cite{Ga1} {\it If $\mathbf{d}=\{ d_n\}$ is a $T$-sequence in an Abelian group $G$, then $\mathbf{n} (G,\mathbf{d})= s_{\mathbf{d}} \left( (G_{d})^\wedge \right)^{\perp}$ (algebraically).}
\end{itemize}

So our third step is as follows: using  Theorem D, we prove that $s_{\mathbf{d}} \left( (G_{d})^\wedge \right)^\perp =H$ in the special cases.
$\Box$

The article is organized as follows. In Section \ref{s2} we prove purely algebraical structural Theorems \ref{td01} and \ref{td02} which  allow us to reduce the proofs of the main theorems to the cases (1)-(4). Theorems \ref{t2} and  \ref{t3}, i.e. the {\it bounded} case,  are proved in Section \ref{s3}.   In the last section we prove  Theorem \ref{t1}, where $H$ is {\it unbounded},  and Theorem \ref{t4}.

\section{Algebraical Theorems} \label{s2}

We will say that a group $X$ {\it satisfies condition $(\Lambda)$} if
$X$ is a finite direct sum of groups of the form $\mathbb{Z} (p^r)^{(\kappa)}$, where $p$ is prime, $r$ is a natural number and the cardinal $\kappa$ is infinite.
In what follows we will use the following trivial corollary of Pr\"{u}fer-Baer's theorem \cite[11.2]{Fuc}.
\begin{lemma} \label{ld01}
Let $G$ be an infinite Abelian group of finite exponent. Then $G$ is the direct sum $G=G_0 \oplus G_1$ of a finite (probably trivial) subgroup $G_0$ and a subgroup $G_1$ satisfying condition $(\Lambda)$.
\end{lemma}

\begin{proof}
By \cite[11.2]{Fuc}, $G$ is a direct sum of cyclic groups, i.e., $G=\bigoplus_{i\in I} \langle g_i \rangle$, where $I$ is infinite. Denote by $J$ the subset of all indices $j\in I$ such that the set $\{ i\in I: \;  o(g_i)=o(g_j)\}$ is finite. If  $J\not= \emptyset$, set $G_0 = \oplus_{j\in J} \langle g_j \rangle$  and $G_1 = \oplus_{i\in I\setminus J} \langle g_i \rangle$. Then $G=G_0 \oplus G_1$. It is clear that $G_0$ is finite and $G_1$ satisfies condition $(\Lambda)$. If  $J$ is empty, then $G=G_1$ satisfies  condition $(\Lambda)$.
\end{proof}

\begin{lemma} \label{ld02}
Let $H=\bigoplus_{i\in I} \langle f_i \rangle$ and $J$ be a non-empty subset of $I$. Assume that  a subset $I'$ of $I\setminus J$ is such that for every $i\in I'$ there is $j(i)\in J$ for which $o(f_{j(i)}) =o(f_i)$. Set
\[
f_i^1 =f_i - f_{j(i)} \mbox{ if } i\in I', \mbox{ and } f_i^1 = f_i \mbox{ for } i\not\in I'.
\]
Then $H=\bigoplus_{i\in I} \langle f_i^1 \rangle$.
\end{lemma}

\begin{proof}
It is enough to show that the set $B=\{ f_i^1 \}_{i\in I}$ is independent. Let
\[
a_1 f_{i_1}^1 +\cdots + a_n f_{i_n}^1 + b_1 f_{j_1}^1 +\cdots + b_m f_{j_m}^1 =0,
\]
where $i_1,\dots, i_n \in I'$ and $j_1,\dots, j_m \in I\setminus I'$ are distinct indices. Then
\[
a_1 f_{i_1} +\cdots + a_n f_{i_n}- a_1 f_{j(i_1)} -\cdots - a_n f_{j(i_n)} + b_1 f_{j_1} +\cdots + b_m f_{j_m} =0.
\]
Since $j(i_1),\dots,j(i_n) \in I\setminus I'$, we must have $a_1 f_{i_1} =\cdots = a_n f_{i_n}=0$. Since $o(f_{j(i_k)}) =o(f_{i_k}), k=1,\dots,n,$ we also have $a_1 f_{j(i_1)} =\cdots = a_n f_{j(i_n)}=0$. Since $f_{j_k}$ are independent, $b_1 f_{j_1}^1 =\cdots = b_m f_{j_m}^1 =0$ either.
\end{proof}

In fact the following proposition  proves case (2) of Theorem \ref{td01}:
\begin{pro} \label{pd01}
{\it Let $G=\mathbb{Z} (p^\infty) +H$, where $H$ is an infinite Abelian group of finite exponent. Then there are a finite (probably trivial) subgroup $H_0$ and an infinite subgroup $H_1$ of $H$ such that
\begin{enumerate}
\item $H=H_0 \oplus H_1$;
\item $G=\left( \mathbb{Z} (p^\infty)+ H_0 \right) \oplus H_1$;
\item $H_1$ satisfies condition $(\Lambda)$.
\end{enumerate} }
\end{pro}

\begin{proof}
{\it Step} 1.
By \cite[Theorem 2.1]{Fuc}, $H$ can be represented as a direct sum of its $p'$-subgroups $H_{p'}$.
Set $Y=\oplus_{p' \not= p} H_{p'}$ and $G_p = \mathbb{Z} (p^\infty) +H_p$. Then $G= G_p \oplus Y$ and $H=H_p \oplus Y$. Since $Y$ has finite exponent, then, by Lemma \ref{ld01}, $Y=Y_0 \oplus Y_1$, where $Y_0$ is finite and $Y_1$  satisfies condition $(\Lambda)$. Thus, if we will find the desired decomposition $H_p = H_0^p \oplus H_1^p$ for $H_p$ and $G_p$, then, setting $H_0 = H_0^p \oplus Y_0$ and $H_1 = H_1^p \oplus Y_1$, we will prove the proposition.

So, it is enough to prove the proposition assuming that $H$ is a $p$-group.

{\it Step} 2. By the Baer theorem \cite[18.1]{Fuc}, there is a subgroup $V$ of $G$ such that $G= \mathbb{Z} (p^\infty)\oplus V$. The natural projections of $G$ onto $\mathbb{Z} (p^\infty)$ and $V$ we denote by $\pi_1$ and $\pi_V$ respectively. Since $V=\pi_V (H)$, the group $V$  also has finite exponent.

By \cite[11.2]{Fuc},  $H=\oplus_{\alpha\in I} \langle e_\alpha \rangle$, where $I$ is an infinite set of indices.
We call two indices $\alpha$ and $\beta$ equivalent if
\begin{equation} \label{11}
o(e_\alpha)=o(e_\beta) \mbox{ and } o\left( \pi_1 (e_\alpha) \right) =o\left( \pi_1 (e_\beta) \right).
\end{equation}
This relation is reflexive, symmetric and transitive. So it is possible to divide $I$ into disjoint classes of equivalent indices. Since $\exp H<\infty$, there is only the finite set of disjoint classes $I_1,\dots, I_s$. Since $\pi_1 (H)$ is a cyclic group, by (\ref{11}), without loss of generality we may assume that $\pi_1 (e_\alpha) =\pi_1 (e_\beta)$ for every $\alpha, \beta \in I_i$ and $1\leq i\leq s$.

For every $1\leq i\leq s$ fix an $\alpha_i \in I_i$ and, for $\alpha\in I_i \setminus \{ \alpha_i\}$, set $v_\alpha =e_\alpha - e_{\alpha_i}$. Then $\pi_1 (v_\alpha)=0$ and hence $v_\alpha \in V$. Set $J=I\setminus \{ \alpha_1,\dots, \alpha_s \}$. Then, by Lemma \ref{ld02}, we have
\[
H= H' \oplus V', \mbox{ where } H' =\oplus_{i=1}^s \langle e_{\alpha_i} \rangle \mbox{ and } V' = \oplus_{\alpha \in J} \langle v_\alpha \rangle  \subseteq V.
\]
In particular, $V= \pi_V (H') + V'$.

(a) Let $\pi_V (H') \cap V' \not= \{ 0\}$. Since $H'$ is finite, there is the minimal finite subset $\{ \beta_1,\dots,\beta_l \} $ of $J$ such that  $\pi_V (H') \cap V' \subseteq R$, where $R=\bigoplus_{i=1}^l \langle v_{\beta_i}\rangle$. Then
\[
H=H'' \oplus V'', \mbox{ where } H'' = H' \oplus R   \mbox{ and } V'' = \bigoplus_{\alpha \in J\setminus \{ \beta_1,\dots,\beta_l \}} \langle v_\alpha \rangle  \subseteq V'.
\]

(b) Let $\pi_V (H') \cap V' = \{ 0\}$. Set $H'' = H', R=0$ and $V'' = V'$.

Therefore in all cases we have
\[
H=H'' \oplus V''  \mbox{ and } G= \left( \mathbb{Z} (p^\infty) +H''\right) + V'',
\]
where $H'' = H' \oplus R$ is finite, $R\subset V$ and $R\cap V'' =\{ 0\}$.

We claim that {\it the last sum is direct}, i.e., $\left( \mathbb{Z} (p^\infty) +H''\right) \cap V'' =\{ 0\}$. Indeed, let $t= f+(h' +r) \in \left( \mathbb{Z} (p^\infty) +H''\right) \cap V''$, where $f\in \mathbb{Z} (p^\infty)$,  $h'\in H'$ and $r\in R$. Then $t=\pi_V (t) =\pi_V (h') +r$ and $\pi_V (h') = t-r \in V'' +R=V'$. So $\pi_V (h') \in \pi_V (H')\cap V' \subseteq R$. Hence $t= \pi_V (h') +r \in R\cap V'' =\{ 0\}$.

By Lemma \ref{ld01}, there is a representation $V''=H'_0 \oplus H_1$, where $H'_0$ is finite and $H_1$ satisfies condition $(\Lambda)$. It is remaind to put $H_0 = H'' \oplus H'_0$.
\end{proof}

\begin{lemma} \label{ld03}
Let an Abelian  $p$-group $G$ have the form $G=\langle A \rangle +H$, where $\exp H<\infty$ and $A=\{ e_i\}_{i=1}^\infty$ is an independent sequence in $G$ such that $o(e_i)\geq \exp H$ for every $i\geq 1$. Then $\langle A \rangle$ is a direct summand of $G$.
\end{lemma}

\begin{proof}
Since $G/\langle A \rangle$ has finite exponent, it is a direct sum of cyclic groups \cite[Theorem 11.2]{Fuc}. By Kulikov's theorem \cite[25.2]{Fuc}, it is enough to show that $\langle A \rangle$ is a pure subgroup of $G$. By \cite[23 I]{Fuc}, we have to prove that $p^l \langle A \rangle = \langle A \rangle \cap p^l G$ for every natural number $l$.

Let $g=p^l g_0 \in \langle A \rangle \cap p^l G$, where $g_0 \in G$. We may assume that $g$ and $g_0$ have the form:
\[
\begin{split}
g & = a_1 p^{l_1} e_{i_1} +\cdots + a_m p^{l_m} e_{i_m} + a_{m+1} p^{l_{m+1}} e_{i_{m+1}} +\cdots + a_n p^{l_n} e_{i_n}, \\
g_0 & = b_1 p^{r_1} e_{i_1} +\cdots + b_m p^{r_m} e_{i_m} + c_{1} p^{s_{1}} e_{j_{1}} +\cdots + c_q p^{s_q} e_{j_q} +h,
\end{split}
\]
where $( a_i,p)= ( b_j,p)= ( c_k,p)=1$ for all $i$, $j$ and $k$, all the indices $i_1,\dots, i_n, j_1,\dots, j_q$ are distinct, and $h\in H$. So
\[
p^l h= \sum_{k=1}^m \left( a_k p^{l_k} - b_k p^{r_k +l}\right) e_{i_k} + \sum_{k=m+1}^n  a_k p^{l_k} e_{i_k} - \sum_{k=1}^q  c_k p^{s_k+l} e_{j_k}.
\]
Since $o(p^l h)\leq \exp H : p^l$ and $o(e_i)\geq \exp H$, we must have $l_k \geq l$ for every $1\leq k\leq n$. In such a case we have
\[
g=p^l g', \mbox{ where } g' = a_1 p^{l_1 -l} e_{i_1} +\cdots  + a_n p^{l_n -l} e_{i_n} \in \langle A \rangle.
\]
\end{proof}

The proof of the next proposition essentially repeats the proof of Proposition \ref{pd01}.
\begin{pro} \label{pd02}
{\it Let an Abelian $p$-group $G$ have the form $G=\langle A \rangle +H$, where $H$ is an uncountable subgroup of finite exponent and $A=\{ e_i\}_{i=1}^\infty$ is an independent sequence in $G$ such that $o(e_i)\geq \exp H$ for every $i\geq 1$. Then there are a countable subgroup $H_0$ and an uncountable subgroup $H_1$ of $H$ such that
\begin{enumerate}
\item $H=H_0 \oplus H_1$;
\item $G=\left( \langle A \rangle + H_0 \right) \oplus H_1$;
\item $H_1$ satisfies condition $(\Lambda)$.
\end{enumerate} }
\end{pro}

\begin{proof}
By Lemma \ref{ld03}, there is a subgroup $V$ of $G$ such that $G= \langle A \rangle \oplus V$. The natural projections of $G$ onto $\langle A \rangle $ and $V$ we denote by $\pi_A$ and $\pi_V$ respectively. Since $V=\pi_V (H)$, the group $V$  also has finite exponent.

By \cite[11.2]{Fuc},  $H=\oplus_{\alpha\in I} \langle e_\alpha \rangle$, where $I$ is an uncountable set of indices.
We call two indices $\alpha$ and $\beta$ equivalent if
\[
o(e_\alpha)=o(e_\beta) \mbox{ and }  \pi_A (e_\alpha)  = \pi_A (e_\beta).
\]
This relation is reflexive, symmetric and transitive. So it is possible to divide $I$ into disjoint classes of equivalent indices. Since $\langle A \rangle $ is countable, there is  at most countable set of disjoint classes $\{ I_i \}_{i=1}^s, s\leq \aleph_0$.

For every $1\leq i\leq s$ fix an $\alpha_i \in I_i$ and, for $\alpha\in I_i \setminus \{ \alpha_i\}$, set $v_\alpha =e_\alpha - e_{\alpha_i}$. Then $\pi_A (v_\alpha)=0$ and hence $v_\alpha \in V$. Set $J=I\setminus \{ \alpha_i\}_{i=1}^s$. Then, by Lemma \ref{ld02}, we have
\[
H= H' \oplus V', \mbox{ where } H' =\oplus_{i=1}^s \langle e_{\alpha_i} \rangle \mbox{ and } V' = \oplus_{\alpha \in J} \langle v_\alpha \rangle  \subseteq V.
\]
In particular, $V= \pi_V (H') + V'$.

(a) Let $\pi_V (H') \cap V' \not= \{ 0\}$. Since $H'$ is countable, there is the minimal countable subset $\{ \beta_i \}_{i=1}^l , l\leq \aleph_0, $ of $J$ such that  $\pi_V (H') \cap V' \subseteq R$, where $R=\bigoplus_{i=1}^l \langle v_{\beta_i}\rangle$. Then
\[
H=H'' \oplus V'', \mbox{ where } H'' = H' \oplus R   \mbox{ and } V'' = \bigoplus_{\alpha \in J\setminus \{ \beta_1,\dots,\beta_l \}} \langle v_\alpha \rangle  \subseteq V'.
\]

(b) Let $\pi_V (H') \cap V' = \{ 0\}$. Set $H'' = H', R=0$ and $V'' = V'$.

Therefore in all cases we have
\[
H=H'' \oplus V''  \mbox{ and } G= \left( \langle A \rangle +H''\right) + V'',
\]
where $H'' = H' \oplus R$ is countable, $R\subset V$ and $R\cap V'' =\{ 0\}$.

We claim that {\it the last sum is direct}, i.e., $\left( \langle A \rangle  +H''\right) \cap V'' =\{ 0\}$. Indeed, let $t= f+(h' +r) \in \left( \langle A \rangle  +H''\right) \cap V''$, where $f\in \langle A \rangle $,  $h'\in H'$ and $r\in R$. Then $t=\pi_V (t) =\pi_V (h') +r$ and $\pi_V (h') = t-r \in V'' +R=V'$. So $\pi_V (h') \in \pi_V (H')\cap V'  \subseteq R$. Hence $t= \pi_V (h') +r \in R\cap V'' =\{ 0\}$.

By Lemma \ref{ld01}, there is a representation $V''=H'_0 \oplus H_1$, where $H'_0$ is finite and $H_1$ satisfies condition $(\Lambda)$. It is remaind to put $H_0 = H'' \oplus H'_0$.
\end{proof}

The following lemma generalized Lemma 3.1 of \cite{Ga3}.
\begin{lemma} \label{ld4}
Let a sequence $\{ b_n\}$  in an Abelian group $G$ be independent and $H$ be a finite subgroup of $G$. Then there is $n_0$ such that $H\cap  \langle b_{n_0}, b_{n_0 +1},\dots \rangle =\{ 0\}$.
\end{lemma}

\begin{proof}
Let $h_1,\dots, h_k$ be an enumeration of all nonzero elements of $H$. For every $1\le i\le k$, by Lemma 3.1 of \cite{Ga3}, we can choose $n_i$ such that the set $\{ h_i,  b_{n_i}, b_{n_i +1},\dots \}$ is independent. Then $n_0 =\max\{ n_1,\dots, n_k \}$ is desired.
\end{proof}

\begin{pro} \label{pd03}
{\it Let $G$ be an Abelian $p$-group of the form $G=\langle A\rangle +H$, where $H$ is a nonzero countable group of finite exponent and $A=\{ g_i\}_{i= 0}^\infty$ is an independent sequence such that either
\begin{enumerate}
\item[(a)] $\exp H \le o(g_0)< o(g_1)<\dots$, or
\item[(b)] $\exp H=o(g_i)$ for every $i\ge 0$.
\end{enumerate}
Then $G$ has a subgroup $G_0$ of the form
\[
G_0=\bigoplus_{i=0}^\infty (H_i + \langle e_i\rangle),
\]
where
\begin{enumerate}
\item the independent sequence $\{ e_i\}$ satisfies respectively one of the conditions $(a)$ or $(b)$;
\item there is $M\le \infty$ such that $H_j$ is a finite nonzero subgroup of $G$ for every $0\le j <M$, and, if $M<\infty$, $H_j = \{ 0\}$ for each $j\ge M$;
\item $H= \bigoplus_{i=0}^\infty H_i$.
\end{enumerate} }
\end{pro}

\begin{proof}
{\it Case 1. $\langle A\rangle \cap H$ is finite.} By Lemma \ref{ld4} we can choose $k\ge 0$ such that $H\cap  \langle g_{k}, g_{k +1},\dots \rangle =\{ 0\}$. Set $e_i = g_{k +i}, i\ge 0$. Let $H=
\bigoplus_{i=0}^{M-1}  \langle h_i\rangle$, where $M\le \aleph_0$ \cite[11.2]{Fuc}. Put $G_0= \langle e_{0}, e_{1},\dots \rangle +H$. Then we have
\[
G_0 =\bigoplus_{i=0}^\infty (H_i \oplus \langle e_i\rangle),
\]
where $H_i =\langle h_i\rangle$ if $i<M$, and $H_i =0$ for $i\ge M$. So $G_0$ is desired.

{\it Case 2. $\langle A\rangle \cap H$ is infinite.} Let $H=\bigoplus_{i=0}^\infty  \langle h_i\rangle$ \cite[11.2]{Fuc}.

We will construct the sequences $\{ H_n\}$ and $\{ e_n\}$ by induction.
Set
\[
G^0 =G,\quad H^0 =H, \; \mbox{ and } \; g_j^0 =g_j, j\ge 0.
\]
Put $ e_0 =g_0^0.$ Choose the minimal index $\kappa_1 \ge 0$ such that
\[
H^0 \cap \langle e_0 \rangle = \left( \bigoplus_{i=0}^{\kappa_1} \langle h_i\rangle \right) \cap \langle e_0 \rangle .
\]
Set
\[
Y_k^1 =\langle \left\{ g^0_{k+i} \right\}_{i=1}^\infty \rangle, k\ge 0, \quad H_0 =\bigoplus_{i=0}^{\kappa_1} \langle h_i\rangle, \; \mbox{ and } X_1 = \bigoplus_{i=\kappa_1 +1}^\infty \langle h_i\rangle.
\]
Then $H^0 = H_0 \oplus X_1$ and
\begin{equation} \label{e03-1}
(H_0 + \langle e_0 \rangle) \cap X_1 =\{ 0\}.
\end{equation}
Indeed, let $ae_0 +h_0 = x$, where $a$ is integer, $h_0 \in H_0$ and $x \in X_1$. Then $ae_0 = x- h_0 \in H^0$ and hence $ae_0 \in H_0$. Thus $x=ae_0 +h_0 \in H_0$ and hence $x=0$.

We distinguish between two cases.

{\it Case 2.1. There is $k\ge 0$ such that}
\[
\left( Y_k^1 + X_1 \right) \cap \left( H_0 + \langle e_0 \rangle \right) =\{ 0\}.
\]
Then we put
\[
H^1 =X_1, \quad g_j^1 =g^0_{k+1 +j}, j\ge 0, \; \mbox{ and } G^1 = \langle \left\{ g^1_j \right\}_{j=0}^\infty \rangle + H^1.
\]
So $(H_0 + \langle e_0 \rangle) \cap G^1 =\{ 0\}$ and we get over to the second step for $G^1, H^1$ and the independent sequence $\left\{ g^1_j \right\}_{j=0}^\infty$ satisfying respectively one of the conditions (a) or (b).

{\it Case 2.2 For every $k\ge 0$,}
\[
\left( Y_k^1 + X_1 \right) \cap \left( H_0 + \langle e_0 \rangle \right) \not= \{ 0\}.
\]
In such a case, because of finiteness of $H_0 + \langle e_0 \rangle$ and since $\exp X_1 <\infty$, we can choose the maximal natural number $m$ satisfying the following condition:
\begin{itemize}
\item[($\alpha$)] there is a nonzero element $h\not= 0$ of $H_0 + \langle e_0 \rangle$ such that for infinitely many indices $k$ there are $y_k \in Y_k^1$ and $z_k \in X_1$ for which
    \[
    y_k + z_k =h \quad \mbox{ and } \quad o(y_k) = p^m.
    \]
\end{itemize}
Fix $h$ satisfying ($\alpha$) and choose a sequence of indices of the form
\begin{equation} \label{e03-2}
0 < i^0_{1} < \cdots < i^0_{s_1} < i^1_{1} < \cdots < i^1_{s_2} < i^2_{1} < \cdots,
\end{equation}
a sequence of integers $a^k_{1}, \dots , a^k_{s_k}, (a_i^j, p)=1, k\ge 0$, a sequence of natural numbers $r^k_{1}, \dots , r^k_{s_k}, k\ge 0$, and a sequence $z_0, z_1, \dots$ in $X_1$ such that, for all $k\ge 0$,
\[
0\not= h=  a^k_{1} p^{r^k_{1}} g_{i^k_{1}}^0 +\cdots + a^k_{s_k} p^{r^k_{s_k}} g_{i^k_{s_k}}^0 + z_k \quad \mbox{ and } \quad o(h-z_k) = p^m.
\]
Set $t_k = \min\{ r^k_{1}, \dots , r^k_{s_k}\}$ and
\[
y'_k = a^k_{1} p^{r^k_{1} -t_k} g_{i^k_{1}}^0 +\cdots + a^k_{s_k} p^{r^k_{s_k}-t_k} g_{i^k_{s_k}}^0, \; k\ge 0.
\]
So $o\left( p^{t_k} y'_k \right) =p^m$ and $o\left(  y'_k \right) =p^{t_k +m}$ for all $k\ge 0$. By (\ref{e03-2}), the sequence $\{ y'_k \}_{k=0}^\infty$ is independent and $p^{t_k} y'_k +z_k =h \in H_0 + \langle e_0 \rangle$ for every $k\ge 0$.

{\it Case 2.2(a)}. Assume that $\exp H \le o(g_0)< o(g_1)<\dots$. Then, by (\ref{e03-2}), $\exp H \le o(y'_0)< o(y'_1)<\dots$, and hence  $t_0 < t_1 < \dots$.
Put
\[
g'_k = p^{t_{2k+1} -t_{2k}} y'_{2k+1} - y'_{2k}, \quad k\ge 0.
\]

{\it Case 2.2(b)}. Assume that $\exp H = o(g_k), \forall k\ge 0$. Then $t_k = t_{k+1}$ and $p^{t_k +m} =\exp H$ for every $k\ge 0$. Put
\[
g'_k = y'_{2k+1} - y'_{2k}, \quad k\ge 0.
\]

In all cases 2.2(a) and 2.2(b) we have the following:
\begin{itemize}
\item[($\alpha_1$)] the sequence $\{ g'_j \}_{j=0}^\infty$ is independent,
\item[($\alpha_2$)] the sequence $\{ g'_j \}_{j=0}^\infty$ satisfies  respectively one of the conditions (a) or (b),
\item[($\alpha_3$)] $o(g'_k) = o(y'_{2k})=p^{t_{2k} +m}$, for every $k\ge 0$,
\item[($\alpha_4$)] $p^{t_{2k}} g'_k = p^{t_{2k+1}} y'_{2k+1} -p^{t_{2k}} y'_{2k} =z_{2k+1} - z_{2k}   \in X_1$.
\end{itemize}
Set $Y'_k = \langle \left\{ g'_j \right\}_{j=k}^\infty \rangle, k\ge 0$.

{\it We claim that there is $k\ge 0$ such that}
\[
\left( Y'_k + X_1 \right) \cap \left( H_0 + \langle e_0 \rangle \right) =\{ 0\}.
\]

Assuming the converse we can find a nonzero element $h'$ of $H_0 + \langle e_0 \rangle$, a sequence of indices of the form
\[
 l^0_{1} < \cdots < l^0_{q_1} < l^1_{1} < \cdots < l^1_{q_2} < l^2_{1} < \cdots,
\]
a sequence of integers $b^k_{1}, \dots , b^k_{q_k}, (b_i^j, p)=1, k\ge 0$, a sequence of natural numbers $w^k_{1}, \dots , w^k_{q_k}, k\ge 0$, and a sequence $x_0, x_1, \dots$ in $X_1$ such that
\[
0\not= h' =  b^k_{1} p^{w^k_{1}} g'_{l^k_{1}} +\cdots + b^k_{q_k} p^{w^k_{q_k}} g'_{l^k_{q_k}} + x_k, \mbox{ for all } k\ge 0.
\]

Assume that  $w_i^k \ge t_{2l_i^k}$ for all $1\le i\le l^k_{q_k}$ beginning from some index $k_0$. Then, by ($\alpha_4$),
\[
0\not= h' =  b^k_{1} p^{w^k_{1}- t_{2l_1^k}}\left( p^{t_{2l_1^k}} g'_{l^k_{1}} \right) +\cdots + b^k_{q_k} p^{w^k_{q_k}-t_{2l_{q_k}^k}} \left( p^{t_{2l_{q_k}^k}}  g'_{l^k_{q_k}}\right) + x_k \in X_1,
\]
for every $k\ge k_0$. This  contradicts to (\ref{e03-1}).

Assume that there is an infinite set $I$ of indices $k$ such that for every $k\in I$ there exists index $1\le i_0^k \le q_k$ for which $w^k_{i_0^k} <t_{2l_{i_0^k}^k}$.
Set $\lambda_k = \min\{ w^k_{1}, \dots , w^k_{q_k}\}$ and  put
\[
y''_k = b^k_{1} p^{w^k_{1}-\lambda_k} g'_{l^k_{1}}  +\cdots + b^k_{q_k} p^{w^k_{q_k}-\lambda_k} g'_{l^k_{q_k}}, \; k\in I.
\]
By construction, $y''_k \in Y^1_k$ for every $k\ge 0$.
Thus, for all $k\in I$, we obtain the following: $y''_k \in Y_k^1,$ $0\not= p^{\lambda_k} y''_k  +x_k  =h' \in H_0 + \langle e_0 \rangle, $ and
\[
\begin{split}
o\left( p^{\lambda_k} y''_k \right) & = \max\left\{  o\left( y'_{2l^k_{1}}\right): p^{w_1^k}, \dots, o\left( y'_{2l^k_{q_k}}\right) : p^{w^k_{q_k}} \right\} \\
& \ge o\left( y'_{2l^k_{i_0^k}}\right) : p^{w^k_{i_0^k}} \ge  o\left( y'_{2l^k_{i_0^k}}\right): p^{t_{2l^k_{i_0^k}} -1} =(\mbox{by } (\alpha_3))= p^{m+1}.
\end{split}
\]
Since $I$ is infinite we obtained a contradiction to the choice of $m$.

Choose $k$ such that $\left( Y'_k + X_1 \right) \cap \left( H_0 + \langle e_0 \rangle \right) =\{ 0\}$. Taking into account ($\alpha_1$) and ($\alpha_2$), we can put
\[
H^1 =X_1, \quad g_j^1 =g'_{k+j}, j\ge 0, \; \mbox{ and } G^1 = \langle \left\{ g^1_j \right\}_{j=0}^\infty \rangle + H^1.
\]
So $(H_0 + \langle e_0 \rangle) \cap G^1 =\{ 0\}$ and we get over to the second step for $G^1, H^1$ and the independent sequence $\left\{ g^1_j \right\}_{j=0}^\infty$  satisfying  respectively one of the conditions (a) or (b).

Iterating this process, we can find a sequence $\{ H_i \}_{i=0}^\infty$ of finite nonzero subgroups of $H$ and an independent sequence $\{ e_i \}$  satisfying  respectively one of the conditions (a) or (b) such that
\[
H=\bigoplus_{i=0}^\infty H_i \; \mbox{ and } \; (H_k + \langle e_k\rangle)\cap \left( \sum_{i=k+1}^\infty ( H_i + \langle e_i\rangle)\right) = \{ 0\}, \mbox{ for every }  k\ge 0.
\]
Hence the sum $G_0 := \sum_{i=0}^\infty ( H_i + \langle e_i\rangle)$ is direct. Thus $G_0$ is desired. This completes the proof of the proposition.
\end{proof}

\begin{theorem} \label{td01}
Let $G$ be an Abelian group of infinite exponent and $H$ its nontrivial subgroup of finite exponent. Then at least one of the following assertions is fulfilled:
\begin{enumerate}
\item[(1)] Assume that $G$ contains an element $g$ of infinite order. Set $G_0 =\langle g\rangle +H$. Then $G_0 \cong \left( \mathbb{Z} \oplus H_0 \right) \oplus X$, where
    \begin{enumerate}
    \item $H_0$ is a finite (probably trivial) subgroup of $H$,
    \item $H=H_0 \oplus X$,
    \item $X\not= \{ 0\}$ if and only if $H$ is infinite. In such a case $X$ satisfies condition $(\Lambda)$.
    \end{enumerate}
\item[(2)] Assume that $G$ contains a subgroup $Y$ of the form $\mathbb{Z} (p^\infty)$. Set $G_0 =Y+H$. Then  $G_0 \cong \left( \mathbb{Z}(p^\infty) + H_0 \right) \oplus X$, where
    \begin{enumerate}
    \item $H_0$ is a finite (probably trivial) subgroup of $H$,
    \item $H=H_0 \oplus X$,
    \item $X\not= \{ 0\}$ if and only if $H$ is infinite. In such a case $X$ satisfies condition $(\Lambda)$.
    \end{enumerate}
\item[(3)] Assume that $G$ contains no neither elements of infinite order nor a subgroup of the form $\mathbb{Z} (p^\infty)$. Then $G$ has a subgroup $G_0$ of the form
\[
G_0=X \oplus \bigoplus_{i=0}^\infty (H_i + \langle e_i\rangle),
\]
where
\begin{enumerate}
\item the independent sequence $\{ e_i\}$ satisfies  the condition
    \[
    \exp H \le o(e_0)<o(e_1)<\dots ;
    \]
\item there is $M\le \infty$ such that $H_j$ is a finite nonzero subgroup of $G$ for every $0\le j <M$, and, if $M<\infty$, $H_j = \{ 0\}$ for each $j\ge M$;
\item $H= X \oplus \bigoplus_{i=0}^\infty H_i$;
\item $X\not= \{ 0\}$ if and only if $H$ is  uncountable. In such a case $X$ satisfies condition $(\Lambda)$.
\end{enumerate}
\end{enumerate}
\end{theorem}

\begin{proof} {\bf (1)} {\it Let $G$ contains an element $g$ of infinite order.} Set $G_0 = \langle g\rangle + H$. It is clear that the sum is direct, i.e., $G_0 = \langle g\rangle \oplus H$. If $H$ is infinite, by Lemma \ref{ld01}, $H$ can be represented in the form $H=H_0 \oplus X$, where $H_0$ is finite and $X$ satisfies condition $(\Lambda)$. So $G_0 \cong \left( \mathbb{Z} \oplus H_0 \right) \oplus X$.

{\bf (2)} {\it Let $G$ contains a subgroup $Y_1$ of the form $\mathbb{Z} (p^\infty)$.} Set $G_0 = Y_1 +H$. Then the assertion follows from Proposition \ref{pd01}.

{\bf (3)} {\it Let $G$ contain no neither elements of infinite order nor a subgroup of the form $\mathbb{Z} (p^\infty)$}.
For a prime $p$ let $H_p$ and $G_p$ be the $p$-component of $H$ and $G$ respectively. Since $H$ is of finite exponent, there are different primes $p_1,\dots, p_n, p_{n+1}, \dots, p_N,$ where $n<\infty$ and $n\le N\le \aleph_0$, such that
\[
H=\bigoplus_{i=1}^n H_{p_i} \mbox{ and } G=\bigoplus_{i=1}^n G_{p_i} \oplus G_1, \mbox{ where } G_1 =\bigoplus_{i=n+1}^N G_{p_i}.
\]

We distinguish between the following two cases.

{\it Case 1}. $\exp G_1 = \infty$.
In such a case, because of $G$ contains no a Pr\"{u}fer group, there is an independent sequence $\{ e_n \}_{n=0}^\infty$ in $G_{1}$, where $\exp H \le o(e_0)< o(e_1)<\dots$. By Lemma \ref{ld01}, $H=H_0 \oplus X$, where $H_0$ is finite (probably trivial) and $X$ satisfies condition $(\Lambda)$. Set
\[
G_0 = \left( (H_0 \oplus \langle e_0 \rangle)\oplus \bigoplus_{i=1}^\infty \langle e_i\rangle \right) \oplus X \mbox{ and } H_i =0, \mbox{ for every } i\ge 1.
\]
Then we obtain the desired.

{\it Case 2.  $\exp G_1 < \infty$.}
In this case there is $1\le l\le n$ such that $\exp G_{p_l} = \infty$. By Lemma \ref{ld01}, $\bigoplus_{i=1, i\not= l}^{n}  H_{p_i}  =H'_0 \oplus X'$, where $H_0$ is finite (probably trivial) and $X'$ satisfies condition $(\Lambda)$. Since $G$ contains no a Pr\"{u}fer group, there is an independent sequence $\{ g_i \}_{i=0}^\infty$ in $G_{p_l}$ satisfying the condition $\exp H \le o(g_0)<o(g_1)<\dots$. Applying  Proposition \ref{pd03}, we can find a subgroup $Y$ of $G_{p_l}$ of the form
\[
Y=X'' \oplus \bigoplus_{i=0}^\infty (H^i_{p_l} + \langle e_i\rangle),
\]
where  $H_{p_l}= X'' \oplus \bigoplus_{i=0}^\infty H^i_{p_l}, |H^i_{p_l}|<\infty$, $X''$  satisfies condition $(\Lambda)$, and  $\exp H\le o(e_0)< o(e_1)<\dots.$
Setting
\[
H^0 =H'_0 \oplus H^0_{p_l}, \; H^i = H^i_{p_l} \mbox{ for } i\ge 1, \; X= X' \oplus X'' \; \mbox{ and } \; G_0 = X\oplus \bigoplus_{i=0}^\infty (H^i + \langle e_i\rangle),
\]
we obtain the desired. The theorem is proved.
\end{proof}

\begin{theorem} \label{td02}
Let $G$ be an Abelian group of finite exponent and $H$ its nonzero subgroup. If $G$ contains a subgroup of the form $\mathbb{Z} (\exp H)^{(\omega)}$, then  $G$ has a subgroup $G_0$ of the form
\[
G_0=X \oplus \bigoplus_{i=0}^\infty (H_i + \langle e_i\rangle),
\]
where
\begin{enumerate}
\item the independent sequence $\{ e_i\}$ satisfies  the condition
    \[
    \exp H = o(e_0)=o(e_1)=\dots ;
    \]
\item there is $M\le \infty$ such that $H_j$ is a finite nonzero subgroup of $G$ for every $0\le j <M$, and, if $M<\infty$, $H_j = \{ 0\}$ for each $j\ge M$;
\item $H= X \oplus \bigoplus_{i=0}^\infty H_i$;
\item $X\not= \{ 0\}$ if and only if $H$ is  uncountable. In such a case $X$ satisfies condition $(\Lambda)$.
\end{enumerate}
\end{theorem}

\begin{proof}
For a prime $p$ let $H'_p$ and $G_p$ be the $p$-component of $H$ and $G$ respectively. Since $G$ has  finite exponent, there are different primes $p_1,\dots, p_n, p_{n+1}, \dots, p_N,$ where $1\le n\le N<\infty$, such that
\[
H=\bigoplus_{i=1}^n H'_{p_i} \mbox{ and } G=\bigoplus_{i=1}^n G_{p_i} \oplus G_1, \mbox{ where } G_1 =\bigoplus_{i=n+1}^N G_{p_i}.
\]
By assumption, for every $1\le i\le n$, $G_{p_i} $ has a subgroup of the form  $\mathbb{Z} (\exp H'_{p_i})^{(\omega)}$.
So, for every $1\le i\le n$, by  Proposition \ref{pd03}, there are an independent sequence $e^i_0, e^i_1, \dots$  in $G_{p_i}$, with $\exp H'_{p_i} = o(e^i_0)=o(e^i_1)=\dots$,  finite (probably trivial beginning from some $M_i \le \aleph_0$) subgroups $H^i_0, H^i_{1}, \dots$ of $H'_{p_i}$ and, if $H'_{p_i}$ is uncountable, an uncountable subgroup $X^i$ of $H'_{p_i}$ such that for the subgroup $G^i_0 := \langle e^i_0, e^i_1,\dots\rangle +H'_{p_i}$ of $G_{p_i}$ we have
    \begin{enumerate}
\item $H'_{p_i}=X^i \oplus \bigoplus_{k=0}^\infty H_k^i $,
\item $G^i_0 = X^i \oplus \bigoplus_{k=0}^\infty (H_k^i + \langle e_k^i\rangle)$,
\item $X^i$ satisfies condition $(\Lambda)$.
    \end{enumerate}
Thus it is enough to put
\[
G_0= \bigoplus_{i=1}^n G^i_0, \; X=\bigoplus_{i=1}^n X^i, \; H_k =\bigoplus_{i=1}^n H^i_k \mbox{ and } e_k =e_k^1 +\cdots + e_k^n \mbox{ for } k\ge 0.
\]
The theorem is proved.
\end{proof}

\section{The bounded case} \label{s3}

In this section we prove Proposition \ref{p00} and Theorems \ref{t2} and \ref{t3}.

{\it Proof of Proposition \ref{p00}}.
Let $G$ be an infinite Abelian group. It is known \cite{CRT} that $G$ admits a non-trivial $TB$-sequence $\mathbf{u}$.  As it was noted in \cite{Ga1}, a sequence $\mathbf{u}$ is a $TB$-sequence if and only if it is a $T$-sequence and $(G, \mathbf{u})$ is maximally almost periodic. So  $\mathbf{n}(G, \mathbf{u})=0$. Thus $G$ admits a  complete non-trivial Hausdorff group topology with  trivial von Neumann radical.
$\Box$

In what follows we need some propositions.

Denote by $S_n =1+\dots +n$. Set $t(n) := \max\{ t:\; n\ge S_t\}$ and $\mu_n = S_{t(n)}$ for every $n\ge 1$. Then $n\left({\rm mod}\, \mu_n \right) <n < S_{n-1}$ for every $n>3$. Also we put $n ({\rm mod}\, 1) =0$ for every natural $n$.

\begin{pro} \label{l5}
{\it Let $G$ have the form $G= \bigoplus_{j=0}^\infty \left( \langle e_j \rangle + H_j \right)$, where \begin{enumerate}
\item there is $M\le \infty$ such that $H_j$ is a finite nonzero subgroup of $G$ with a basis $h^j_1,\dots, h^j_{a_j}$ for every $0\le j <M$, and, if $M<\infty$, $H_j = \{ 0\}$ for each $j\ge M$;
\item the subgroup $H:= \bigoplus_{j=0}^\infty H_j$ of $G$ has finite exponent;
\item $u_j := o(e_j)$ is finite for every $j\ge 0$.
\end{enumerate}
Enumerate all $h^j_i$ (they form a basis of $H$), where $1\le i\le a_j$ and $0\le j <M$, consistently setting:
\[
b_0 = h^0_1, \dots, b_{a_0 -1} = h^0_{a_0},  b_{a_0} = h^1_1, \dots, b_{a_0 +a_1 -1} = h^1_{a_1},  b_{a_0 + a_1} = h^2_{1}, \dots
\]
Suppose one of the following conditions is fulfilled:
\begin{enumerate}
\item[a)] $u_{0} =u_{1} =\dots :=u$ and $u$ is divided on $\exp H$;
\item[b)] $\exp H \le u_j$ for each $j\ge 0$, and $u_j \to \infty$.
\end{enumerate}
Denote by $\mathbf{d}=\{ d_n \} (n\ge 0)$ the following sequence:
\begin{enumerate}
\item[(i)] if $M<\infty$ and setting $c:= a_0 + \cdots +a_{M-1}$,
\[
\begin{split}
 d_{2n} & :  \; e_0, 2 e_0, \dots, (u_0 -1) e_0, \; e_{1}, 2 e_{1}, \dots, (u_{1} -1) e_{1}, \dots \\
 d_{1} & = b_0,   \; d_3= b_0 + e_1, \\
 d_{2n+1} & = b_{n\left({\rm mod}\, c \right) } + e_{S_{n-1} +1} +e_{S_{n-1} +2} + \dots + e_{S_{n}}, \; (n>1);
\end{split}
\]
\item[(ii)] if $M=\infty$,
\[
\begin{split}
 d_{2n} & :  \; e_0, 2 e_0, \dots, (u_0 -1) e_0, \; e_{1}, 2 e_{1}, \dots, (u_{1} -1) e_{1}, \dots \\ d_{1} & = b_0,   \; d_3= b_0 + e_1, \; d_5 = b_1 +(e_2 + e_3), \\
 d_{2n+1} & = b_{n\left({\rm mod}\, \mu_n \right) } + e_{S_{n-1} +1} +e_{S_{n-1} +2} + \dots + e_{S_{n}}, \; (n>2).
\end{split}
\]
\end{enumerate}
Then $\mathbf{d}$ is a $T$-sequence in $G$.}
\end{pro}

\begin{proof}  In both cases (i) and (ii) the proof is the same, so we prove the proposition only for the case ``$M=\infty$''.

Let $k\ge 0$ and
\[
g=(\lambda_1 e_{v_1}+h_{v_1}) + (\lambda_2 e_{v_2}+h_{v_2}) +\dots +(\lambda_q e_{v_q}+h_{v_q}) \not=0, \]
where  $v_1 <\dots <v_q$, $0\le \lambda_j < u_{v_j}$ and $h_{v_j}\in H_{v_j}$ for every $1\le j\le q$.
We need to show that the condition of the Protasov-Zelenyuk criterion is fulfilled, i.e., there exists a natural $m$ such that $g\not\in  A(k,m)$. By the construction of $\mathbf{d}$, there is $m' >3$ such that $d_{2n} = \lambda(n) e_{r(n)}$, where $r(n) > \max \{ v_q , 3\}$ for every $n>m'$. Assume that $g\in A(k,2m_0)$ for some $m_0 >m'$. Then
\[
g=  l_1 d_{2r_1 +1} + l_2 d_{2r_2 +1} +\dots + l_s d_{2r_s +1} +
 l_{s+1} d_{2r_{s+1}} +l_{s+2} d_{2r_{s+2}} +\dots + l_{h} d_{2r_{h}},
\]
where all summands are nonzero, $\sum_{i=1}^h |l_i | \le k+1$, $0<s\le h$ (by our choice of $m_0$) and
\[
\begin{split}
{} &  2m_0 < 2r_1 +1 < 2r_2 +1 < \dots < 2r_s +1 ,    \\
& 2m_0 \le 2r_{s+1} < 2r_{s+2} < \dots < 2r_{h}.
\end{split}
\]
Moreover, by construction, all the elements $d_{2n+1} - b_{n\left({\rm mod}\, \mu_n \right) }$ are independent. So, by the construction of the elements $d_{2n}$ and since $n\left({\rm mod}\, \mu_n \right) < S_{n-1}$ for every $n>3$, there is a subset $\Omega$ of the set $\{ s+1,\dots, h \}$ such that
\begin{equation} \label{7}
\begin{split}
& l_s (d_{2r_s +1} -  b_{r_s\left({\rm mod}\, \mu_{r_s} \right) })  +  \sum_{w \in \Omega} l_w d_{2r_w} \\
& = l_s (e_{S_{r_s -1} +1} +e_{S_{r_s -1} +2} + \dots + e_{S_{r_s}}) +\sum_{w \in \Omega} l_w d_{2r_w}=0.
\end{split}
\end{equation}

a) {\it Assume that $u_{0} =u_{1} =\dots =u$ and $u$ is divided on $\exp H$}. Set $m_0 = 4m'(k+1)$. Then $d_{2r_s +1} - b_{r_s\left({\rm mod}\, \mu_{r_s} \right) }$ contains exactly $r_s > m_0 -1 \ge 4k+4$
independent summands of the form $e_j$.
Since $l_s d_{2r_s +1}\not= 0$ and $u$ is divided on $\exp H$,  $l_s$ is not divided on $u$. So $l_s (d_{2r_s +1} -  b_{r_s\left({\rm mod}\, \mu_{r_s} \right) })$ contains  $r_s >4k+4$ non-zero independent  summands of the form $l_s e_j$. Since $h-s \le k$  and $l_w d_{2r_w}$ has the form $\lambda_j e_j$, the equality (\ref{7}) is impossible. Thus $g\not\in A(k,2m_0)$. So $\mathbf{d}$ is a $T$-sequence.

b) {\it Assume that $\exp H \le u_j$ for each $j\ge 0$ and $u_j \to \infty$}.  Choose $j' >m'$ such that $u_{j} > 2(k+1)$ for every $j> j'$. Set $m_0 = 4j'(k+1)$.
Then $d_{2r_s +1} - b_{r_s\left({\rm mod}\, \mu_{r_s} \right) }$ contains exactly $r_s > m_0 -1 \ge 4k+4$ summands of the form $e_{j}$ with $j\ge S_{r_s -1} +1 > m_0 > j'$. So, since $|l_s| \le k+1$, $l_s d_{2r_s -1}$ contains  $r_s >4(k+1)$ non-zero independent  summands of the form $l_s e_{j}$. Since $h-s \le k$ and $l_w d_{2r_w}$ has the form $a_j e_j$, the equality (\ref{7}) is impossible. Thus $g\not\in A(k,2m_0)$. So $\mathbf{d}$ is a $T$-sequence.
\end{proof}

\begin{lemma} \label{l6}
Let $G= \bigoplus_{j=0}^\infty G_j$, where $G_j$ is nonzero for all $j\ge 0$, be endowed with the discrete topology and let $H_j$ be a subgroup of $G_j$ for every $0\le j <\infty$.  Set $H= \bigoplus_{j=0}^\infty H_j$ and $Y= \bigoplus_{j=0}^\infty  H_j^\perp$, where $H_j^\perp$ is the annihilator of $H_j $ in $G_j^\wedge$. Then $H^\perp = \prod_{j=0}^\infty H_j^\perp$ and $Y$ is a dense subgroup of $H^\perp$.
\end{lemma}

\begin{proof}  Since $g=(g_j) \in H^\perp$ iff $(g, h_j)=(g_j, h_j)=1$ for every $h_j \in H_j$ and every $j\ge 0$, we obtain that $H^\perp = \prod_{j=0}^\infty H_j^\perp$. So $Y\subset H^\perp$ is a dense subgroup of $H^\perp$.
\end{proof}

\begin{pro} \label{p3}
{\it Let $G$ have the form $G= \bigoplus_{j=0}^\infty \left( \langle e_j \rangle + H_j \right)$, where \begin{enumerate}
\item there is $M\le \infty$ such that $H_j$ is a finite nonzero subgroup of $G$ for every $0\le j <M$, and, if $M<\infty$, $H_j = \{ 0\}$ for each $j\ge M$;
\item the subgroup $H:= \bigoplus_{j=0}^\infty H_j$ of $G$ has finite exponent;
\item $u_j := o(e_j)$ is finite for every $j\ge 0$.
\end{enumerate}
Suppose that  one of the following conditions holds
\begin{enumerate}
\item[(a)] $u_{0}=u_1 =\dots :=u$ and $u$ is divided on $\exp H$,
\item[(b)] $\exp H \le u_j$ for each $j\ge 0$ and $u_j \to \infty$.
\end{enumerate}
Then $G$ admits a complete Hausdorff group topology $\tau$ such that $\mathbf{n}(G,\tau) = H$.}
\end{pro}

\begin{proof}
We prove the proposition for the case ``$M=\infty$''.

Let $\mathbf{d}=\{ d_n \}$ be the $T$-sequence which was defined in Proposition \ref{l5}. By Theorem D and Lemma \ref{l6}, it is enough to show that
$\bigoplus_{j=0}^\infty H_j^\perp \subseteq s_{\mathbf{d}} ((G_{d})^\wedge)) \subseteq H^\perp$.

The inclusion $\bigoplus_{j=0}^\infty H_j^\perp \subseteq s_{\mathbf{d}} ((G_{d})^\wedge))$ is trivial. Let us show the second one.
Let $\omega =(g_0, g_1, \dots)\in s_{\mathbf{d}} ((G_{d})^\wedge)$. By definition, there exists $N\in \mathbb{N}$ such that $|1- (d_{2n}, \omega)|<0.1, \forall n>N$. Thus, there is $N_0 >N$ such that $|1- ( l e_{j}, \omega)|= |1- (l e_{j}, g_j)| <0.1, \forall l=1,\dots, u_j -1,$ for every $j>N_0$. This means that $g_j \in \langle e_j \rangle^\perp$ for every $j>N_0$.  For every $j\ge 0$ and each $1\le i\le a_j$, put $k=a_0 +\dots +a_{j-1} +i$ (we assume that $a_{-1}=0$) and set $\mu_n = S_{t(n)}$. Since  $(d_{2(\mu_n +k)+1}, \omega) \to 1$ at $n\to\infty$ too and $(d_{2(\mu_n +k)+1}, \omega) = (h_i^j, g_j)$ for all sufficiently large $n$, we obtain that $g_j \in \langle h_i^j \rangle^\perp$ for any $j\ge 0$ and  each $1\le i\le a_j$. So $g_j \in H_j^\perp$ for every $j\ge 0$. Thus $ s_{\mathbf{d}} ((G_{d})^\wedge) \subseteq H^\perp$ by Lemma \ref{l6}.
\end{proof}

{\it Proof  of Proposition \ref{p4}}. Let us endow the group $\bigoplus_{\alpha \in I} G_\alpha$ under the asterisk group topology \cite{Kap} and $\prod_{\beta \in J} H_\beta$ under the product group topology. By the Kaplan theorem \cite{Kap}, $G^\wedge = \prod_{\alpha \in I} G_\alpha^\wedge \times \bigoplus_{\beta \in J} H_\beta^\wedge$. Since $G_\alpha^\wedge$ and  $H_\beta^\wedge$ are trivial for all $\alpha\in I$ and $\beta\in J$, then  $G^\wedge$ is also trivial.
$\Box$

\begin{cor} \label{co1}
{\it Let an Abelian group $X$ satisfy condition $(\Lambda)$. Then $X$ admits a (complete, if $X$ is countably infinite) MinAP group topology.}
\end{cor}

\begin{proof}
By Proposition \ref{p4}, it is enough to show that a group $Y$ of the form $\bigoplus_{i=0}^\infty \langle e_i \rangle$, where $o(e_0)=o(e_1)=\dots$, admits a complete MinAP group topology. But this immediately follows from Proposition \ref{p3} if we put $H_i :=\langle e_i \rangle $ for all $i\ge 0$.
\end{proof}

{\it Proof of Theorem \ref{t2}}. {\it Necessity.} We essentially follow  the arguments of Remus (see \cite{Com}). Let
\[
\exp H =p_1^{b_1}  \dots  p_l^{b_l} \mbox{ and } \exp G =p_1^{n_1}  \dots  p_l^{n_l} \cdot p_{l+1}^{n_{l+1}}  \dots  p_t^{n_t},
\]
where $p_1, \dots, p_t$ are distinct prime integers.

Assuming the converse, we obtain that there is a $1\le j \le l$ such that $G$ contains only a finite subset of independent elements $g$ for which $o(g) = p_j^a$ with $a\ge b_j$. Set $m := \exp G /p_{j}^{n_j - b_j+1}$ and $\pi : G\to G, \, \pi (g) = m g$. Then $\pi (H) \not= 0$ and, by our hypotheses, $\pi (G)$ is finite.

Now let $\tau$ be an arbitrary Hausdorff group topology on $G$. Since $\pi (G)$ is finite, ${\rm Ker} (\pi)$ is open and closed and, hence, dually closed and dually embedded \cite{Nob}. So, by Lemma \ref{l1}, $\mathbf{n} (G,\tau) \subseteq {\rm Ker} (\pi)$. Hence $H\not= \mathbf{n} (G,\tau)$.
This is a contradiction.

{\it Sufficiency}. By Theorem \ref{td02} and the reduction principle, we can restrict ourselves only to the case
\[
H=X \oplus \bigoplus_{i=0}^\infty H_i, \quad G = X \oplus \bigoplus_{i=0}^\infty (H_i + \langle e_i\rangle),
\]
where  $X$ satisfies condition $(\Lambda)$ and $o(e_i)=\exp H$ for every $i\ge 0$. Taking into account that the von Neumann radical of a product of topological groups is the product of their von Neumann radicals, $G$ admits a (complete if $H$ is countable) Hausdorff group topology $\tau$ such that $\mathbf{n}(G,\tau) = H$ by Proposition \ref{p3} and Corollary \ref{co1}. The theorem is proved.
$\Box$

{\it Proof of Theorem \ref{t3}}. The theorem immediately follows from Theorems \ref{td01} and B, the reduction principle, Proposition \ref{p3}, and Corollary \ref{co1}.
$\Box$

\section{The unbounded case} \label{s4}

We start with the following simple observation:

\begin{pro} \label{p6}
{\it Let a countably infinite Abelian group $G$ embed into a compact metrizable Abelian group $X$ with the dense image and $\pi : G\to X$ be an embedding. Let a sequence $\mathbf{d}=\{ d_n\}$ in $G$ be such that $\pi^\wedge (X^\wedge)\subseteq s_{\mathbf{d}} \left( (G_{d})^\wedge \right)$, where $\pi^\wedge$ is the conjugate homomorphism of $\pi$. Then $\pi (d_n)$ converges to zero in $X$. In particular, $\mathbf{d}$ is a $TB$-sequence.}
\end{pro}

\begin{proof}   By definition, we have $(\pi (d_n), y)= (d_n, \pi^\wedge (y))\to 1$ for every $y\in X^\wedge$. But this is possible only if $\pi (d_n) \to 0$.
\end{proof}

The following proposition  is an easy generalization of Theorem 2.1.5 of \cite{ZP2}:
\begin{pro} \label{p5}
{\it Let $G$ be a subgroup of an Abelian Hausdorff topological group $S$. Assume that $q$ sequences $\{ a_n^j \}_{n\in \mathbb{N}}, 0\le j\le q-1$, in $G$ are such that $\{ a_n^j \}_{n\in \mathbb{N}}$ converges to an element $b^j \in S$. If $\langle b^0, \dots, b^{q-1}  \rangle \cap G= \{ 0\}$, then the following sequence
\[
d_{nq+j} = a_{n}^j, \mbox{ for every } n\ge 1 \mbox{ and } 0\le j\le q-1,
\]
is a $T$-sequence in $G$.}
\end{pro}

\begin{proof}  Let $0\not= g\in G$ and $k\ge 0$. Since $\langle b^0, \dots, b^{q-1}  \rangle \cap G= \{ 0\}$, we can choose a symmetric neighborhood $U$ of zero in $S$ such that
\[
g\not\in \bigcup_{ |t_j| <k+2, \, 0\le j\le q-1} (t_0 b^0 + \dots + t_{q-1} b^{q-1} +U).
\]
Let $V$ be a symmetric neighborhood of zero in $S$ such that $(k+1)V\subset U$. Choose a $m>0$ such that $a_n^j \in b^j + V$ for all $n\ge m$ and $0\le j\le q-1$. Set $m_0 =(m+1)q$. Then for every $l_1,\dots, l_s$ such that $\sum_{i=1}^s |l_i| \le k+1$ and for any $r_1 <\dots <r_s, m_0 < r_1,$ we have
\[
\begin{split}
{} & l_1 d_{r_1}  + \dots + l_s d_{r_s} = l_1 a_{\frac{1}{q} (r_1 - r_1 ({\rm mod}\, q))}^{r_1 ({\rm mod}\, q)} + l_2 a_{\frac{1}{q} (r_2 - r_2 ({\rm mod}\, q))}^{r_2 ({\rm mod}\, q)} +\dots + l_1 a_{\frac{1}{q} (r_s - r_s ({\rm mod}\, q))}^{r_s ({\rm mod}\, q)}  \\
{} & \in l_1 b^{r_1 ({\rm mod}\, q)} + (|l_1|)V + l_2 b^{r_2 ({\rm mod}\, q)} + (|l_2|)V + \dots + l_s b^{r_s ({\rm mod}\, q)} + (|l_s|)V \\
{} & \subset  \bigcup_{ |t_j| <k+2, \, 0\le j\le q-1} (t_0 b^0 + \dots + t_{q-1} b^{q-1} +U).
\end{split}
\]
Hence $g \not\in A(k,m_0 )$ and the condition of the Protasov-Zelenyuk criterion (Theorem B) holds. So $\mathbf{d}=\{ d_n \}$ is a $T$-sequence.
\end{proof}

\begin{cor}  \cite[Theorem 2.1.5]{ZP2} \label{co3}
{\it Let $G$ be a subgroup of an Abelian Hausdorff topological group $S$.  If a sequence $\{ a_n\}$ converges to zero and $\{ b_n\}$ converges to an element $b \in S$ satisfying the condition $\langle b \rangle \cap G= \{ 0\}$. Then the following sequence
\[
d_{2n} = a_{n} \mbox{ and } d_{2n+1} = b_{n}, \mbox{ for every } n\ge 1,
\]
is a $T$-sequence in $G$. }
\end{cor}

The following two theorems are very important for our considerations:
\begin{itemize}
\item[] {\bf Theorem E.} \cite{DiK} {\it For every countably infinite subgroup $G$ of a compact metrizable Abelian group $X$ there is a sequence $\mathbf{d}=\{ d_n\}$ in $X^\wedge$ such that $G= s_{\mathbf{d}} (X)$.}
\item[] {\bf Theorem F.} \cite{DiS} {\it Every countably infinite unbounded Abelian group $G$ is isomorphic to a dense subgroup of the connected  second countable group $\mathbb{T}^\omega$.}
\end{itemize}

\begin{theorem} \label{t6}
Let a countably infinite Abelian group $G$ embed into a compact connected second countable group $X$ with the dense image. Then $G$  admits a complete Hausdorff minimally almost periodic group topology.
\end{theorem}

\begin{proof} Let $\pi : G\to X$ be an embedding.
By Theorem E, we can choose a sequence $\mathbf{a}=\{ a_n\}$ in $G$ such that $s_{\mathbf{a}} \left( (G_{d})^\wedge \right) = \pi^\wedge (X^\wedge)$. By Proposition \ref{p6}, $\pi (a_n)\to 0$ in $X$.

Denote by $\Omega$ the set of all elements $x$ such that $\langle x\rangle$ is dense in $X$. Then $m_X (\Omega) =1$, where $m_X$ is the normalized Haar measure on $X$  \cite[25.27]{HR1}. Thus $\Omega$ is continual and hence there exists $b\in \Omega$ such that $\pi (G)\cap \langle b\rangle = \{ 0\}$. Let a sequence $\mathbf{b} =\{ b_n\}$ in $G$ be such that  $\pi (b_n)$ converges to $b$ in $X$. Then, by Corollary \ref{co3}, the sequence
\[
d_{2n} = a_{n} \mbox{ and } d_{2n+1} = b_{n}, \mbox{ for every } n\ge 1,
\]
is a $T$-sequence in $G$.

By Theorem  D, it is enough to prove that $s_{\mathbf{d}} \left( (G_{d})^\wedge \right) = \{ 0\}$. Let $x\in s_{\mathbf{d}} \left( (G_{d})^\wedge \right)$. Then $(d_{2n}, x) = (a_{n}, x) \to 1$. So, by our choice of the sequence $\mathbf{a}$, $x= \pi^\wedge (y)$ for some $y\in X^\wedge$. Thus
$(b_n, \pi^\wedge (y)) = (\pi (b_n), y) \to (b,y)$. Since also $(d_{2n+1}, x)= (b_n, \pi^\wedge (y))\to 1$, we have $(b,y)=1$. Hence $(k b ,y)=1$ for every $k\in \mathbb{Z}$. Since $\langle b\rangle$ is dense in $X$, we must have $y=0$. So $x=0$.
\end{proof}

{\it Proof of Theorem \ref{t1}}. If $H$ is bounded, the theorem follows from Theorem \ref{t2}. If $H$ is  unbounded, the theorem immediately follows from Theorems F and \ref{t6}.
$\Box$

{\it Proof of Theorem \ref{t4}}. Let  $\Omega_1$ be the set of all indices $i$ such that $G_i$ is bounded and $\Omega_2$ be the set of all  $i$ for which $G_i$ is unbounded.

{\it Case 1.} $\Omega_1 =\emptyset$. By Theorem \ref{t1}, any $G_i$ admits a MinAP group topology. Thus $G$ also admits a MinAP group topology by Proposition \ref{p4}.

{\it Case 2.}  $\Omega_1 \not=\emptyset$. For every $i\in \Omega_1$, the group $G_i$ is a direct sum of finite cyclic groups \cite[11.2]{Fuc}:
\[
G_i = \bigoplus_{(p,n)\in A_i}  \mathbb{Z}(p^n)^{(k_{(p,n)}^i)},
\]
where $p\in P$, $n\in \mathbb{N}$ and $A_i$ is finite.

Set $C= \cup_{i\in \Omega_1} A_i$ and $k_{(p,n)} =\sum_{i\in \Omega_1} k_{(p,n)}^i $. Put $C_1 := \{ (p,n) \in C : \; k_{(p,n)} <\infty \}$ and $C_2 = C \setminus C_1$. Then we can represent the group $G$ in the following form
\[
G = \left( \bigoplus_{(p,n)\in C_1} \mathbb{Z} (p^n)^{(k_{(p,n)})} \right) \oplus \bigoplus_{(p,n)\in C_2} \mathbb{Z} (p^n)^{(k_{(p,n)})} \oplus \bigoplus_{i\in \Omega_2} G_i .
\]
Denote the first summand by $H$.

{\it Subcase $2(a)$.} Assume that $H$ is either countably infinite (and hence unbounded) or trivial. Then, by Corollaries \ref{c1} and \ref{c2} and Proposition \ref{p4}, $G$ admits a MinAP group topology.

{\it Subcase $2(b)$.} Let $H$ be nonzero and finite. If $\Omega_2 \not=\emptyset$, then we have
\[
G= (H\oplus G_{i_0} ) \oplus \bigoplus_{(p,n)\in C_2} \mathbb{Z} (p^n)^{(k_{(p,n)})} \oplus \bigoplus_{i\in \Omega_2, i\not= i_0} G_i ,
\]
and also, by Corollaries \ref{c1} and \ref{c2} and Proposition \ref{p4}, $G$ admits a MinAP group topology.

If $\Omega_2 =\emptyset$, then, since $G$ is unbounded, $C_2$ is infinite. So
\[
G= \left( \bigoplus_{(p,n)\in C_1} \mathbb{Z} (p^n)^{(k_{(p,n)})} \oplus \bigoplus_{(p,n)\in C_2} \mathbb{Z} (p^n)  \right) \oplus \bigoplus_{(p,n)\in C_2} \mathbb{Z} (p^n)^{(k_{(p,n)} -1)},
\]
and, by Corollaries \ref{c1} and \ref{c2} and Proposition \ref{p4}, $G$ admits a MinAP group topology too. $\Box$

\begin{rem} \label{r1}
{\it Let us show that  Question 7 can be reduced to this question  only for the special groups in it.} By Theorem 25.25 of \cite{HR1}, every infinite compact Abelian group $G$ is algebraically isomorphic to a group of the form
\[
\prod_{p\in P} \left[ \Delta_p^{\mathfrak{a}_p} \times \prod_{n=1}^\infty \left(\mathbb{Z} (p^{n} ) \right)^{k_{n,p}} \right] \times \bigoplus_{p\in P} \mathbb{Z}(p^\infty)^{(\mathfrak{b}_p)} \times \mathbb{Q}^{(\mathfrak{n})}.
\]
Set $E_1 = \{ (n,p) :\; 0< k_{n,p} <\infty \}$ and $E_2 = \{ (n,p) :\;  k_{n,p}=\infty \}$. Then $G$ has the form
\[
\left(\prod_{(n,p) \in E_1} \left(\mathbb{Z} (p^{n} ) \right)^{k_{n,p}} \right) \times \prod_{p\in P} \Delta_p^{\mathfrak{a}_p} \times \prod_{(n,p) \in E_2} \left(\mathbb{Z} (p^{n} ) \right)^{k_{n,p}}  \times  \bigoplus_{p\in P} \mathbb{Z}(p^\infty)^{(\mathfrak{b}_p)} \times \mathbb{Q}^{(\mathfrak{n})}.
\]

{\it Case} 1.   $|E_1 |=\infty$.
Then $G$ is a product of the special groups in the question, $\left(\mathbb{Z} (p^{n} ) \right)^{k_{n,p}}, (n,p) \in E_2, $  and a sum of unbounded countable groups   $\mathbb{Z}(p^\infty)$ and the full rational groups $\mathbb{Q}$. By Theorem \ref{t2} and \ref{t1}, all groups $\left(\mathbb{Z} (p^{n} ) \right)^{k_{n,p}}, (n,p) \in E_2 $, $\mathbb{Z}(p^\infty)$ and $\mathbb{Q}$ admit a MinAP group topology. So the assertion follows from Proposition \ref{p4}.

{\it Case} 2. {\it $E_1$ is finite}. Then either at least one of the numbers $\mathfrak{a}_p , \mathfrak{b}_p$ or $\mathfrak{n}$  is nonzero or $| E_2 |= \infty$. So we can adjoin the first finite group $\prod_{(n,p) \in E_1} \left(\mathbb{Z} (p^{n} ) \right)^{k_{n,p}}$  either to one of the group $\Delta_p$, $\mathbb{Z}(p^\infty)$, $\mathbb{Q}$ or to  the group $\prod_{(n,p) \in E_2} \mathbb{Z} (p^{n})$. Therefore also in this case, if the special groups admit a MinAP group topology, then $G$ admits a MinAP group topology too. $\Box$
\end{rem}

\bibliographystyle{amsplain}

\end{document}